\documentclass[10pt]{amsart}

\usepackage[top=3cm,bottom=3cm,left=3.5cm,right=3.5cm,marginparsep=0.3cm,marginparwidth=3cm,includefoot]{geometry}
\usepackage[foot]{amsaddr}
\usepackage{amssymb}
\usepackage{mathabx}
\usepackage[backend = biber]{biblatex}
\usepackage{bbm}
\usepackage[english]{babel}
\usepackage[utf8]{inputenc}
\usepackage[T1]{fontenc}
\usepackage{hyperref}
\hypersetup{
    colorlinks = true,
    linkcolor = purple,
    citecolor = brown,
    urlcolor = blue
    }
\usepackage{ifpdf}
\ifpdf
  \usepackage[pdftex]{graphicx}
  \usepackage{epstopdf}
\else
  \usepackage[dvips]{graphicx}
\fi
\usepackage{graphicx}
\usepackage{xcolor}
\usepackage{cleveref}
\usepackage[all]{xy}
\usepackage{csquotes}
\usepackage{dynkin-diagrams}

\makeatother

\theoremstyle{plain}
\newtheorem{theorem}{Theorem}[section]

\newtheorem{lemma}[theorem]{Lemma}
\newtheorem{proposition}[theorem]{Proposition}
\newtheorem{corollary}[theorem]{Corollary}
\newtheorem{question}[theorem]{Question}
\newtheorem{conjecture}[theorem]{Conjecture}

\theoremstyle{definition}
\newtheorem{definition}[theorem]{Definition}

\newtheorem{convention}[theorem]{Convention}
\newtheorem{notation}[theorem]{Notation}
\newtheorem{remark}[theorem]{Remark}

\theoremstyle{remark}
\newtheorem{example}[theorem]{Example}

\newcommand{\I}{\mathcal{I}}
\newcommand{\F}{\mathcal{F}}
\newcommand{\alg}{\mathbb{R}_\mathrm{alg}}
\newcommand{\anexp}{\mathbb{R}_\mathrm{an,exp}}
\newcommand{\poinca}{\mathbb{H}}

\newcommand{\Z}{\mathbb Z}
\newcommand{\Q}{\mathbb Q}
\newcommand{\C}{\mathbb C}
\newcommand{\R}{\mathbb R}

\newcommand{\V}{\mathbb V}

\newcommand{\p}{\mathbb P}
\newcommand{\s}{\mathbb S}
\newcommand{\sm}{\mathrm{sm}}

\newcommand{\ppabel}{\mathcal{A}}
\newcommand{\cur}{\mathcal{M}}
\newcommand{\torelli}{\mathcal{T}}
\newcommand{\halfplane}{\mathbb H}

\newcommand{\mult}{\mathbb G}
\newcommand{\G}{\mathbf G}
\newcommand{\mon}{\mathbf H}
\newcommand{\M}{\mathbf M}
\newcommand{\LL}{\mathbf L}
\newcommand{\NN}{\mathbf N}
\newcommand{\MT}{\mathbf{MT}}

\newcommand{\HL}{\mathrm{HL}(S, \V^\otimes)}
\newcommand{\HLM}{\mathrm{HL}(S, \V^\otimes, \M)}

\newcommand{\pr}{\mathrm{pr}}
\newcommand{\an}{\mathrm{an}}

\newcommand{\Hom}{\mathrm{Hom}}

\newcommand{\h}{\mathfrak{h}}
\newcommand{\quo}{\backslash}
\newcommand{\codim}{\mathrm{codim}}

\newcommand{\gl}{\mathbf{GL}}

\newcommand{\gsp}{\mathbf{GSp}}
\newcommand{\ssp}{\mathbf{Sp}}

\newcommand{\res}{\mathrm{Res}}

\newcommand{\ad}{\mathrm{ad}}

\newcommand{\typ}{\mathrm{typ}}
\newcommand{\atyp}{\mathrm{atyp}}

\newcommand{\der}{\mathrm{der}}

\title{Existence and Density of typical Hodge Loci}
\author{Nazim Khelifa}
\author{David Urbanik}
\date{\today}

\begin{document}
\maketitle
\begin{abstract}
Motivated by a question of Baldi-Klingler-Ullmo, we provide a general sufficient criterion for the existence and analytic density of typical Hodge loci associated to a polarizable $\mathbb{Z}$-variation of Hodge structures $\mathbb{V}$. Our criterion reproves the existing results in the literature on density of Noether-Lefschetz loci. It also applies to understand Hodge loci of subvarieties of $\mathcal{A}_{g}$. For instance, we prove that for $g \geqslant 4$, if a subvariety $S$ of $\mathcal{A}_{g}$ has dimension at least $g$ then it has an analytically dense typical Hodge locus. This applies for example to the Torelli locus of $\mathcal{A}_g$. . 
\end{abstract}
\begin{center}
\tableofcontents
\end{center}
\section{Introduction}
\subsection{Context} The observation that Hodge loci of polarizable $\Z$-variations of Hodge structures ($\Z$-VHS) can be realized as intersection loci has recently been fruitfully applied to obtain refinements of the celebrated algebraicity result of Cattani-Deligne-Kaplan \cite{cdk, bkt}. For instance, if $\V$ is a $\Z$-VHS on a smooth, irreducible and quasi-projective variety $S$ over $\C$, one would like to understand criteria for density, in both the Zariski and analytic topology, for families of components of the Hodge locus $\HL$. With this goal in mind Klingler \cite{klingler_atyp} introduced a dichotomy, later refined by Baldi-Klingler-Ullmo, between the part of the Hodge locus corresponding to typical intersections and the part corresponding to atypical ones (\cite[Def. 2.2, 2.4]{bku}):
\[
\HL = \HL_\typ \cup \HL_\atyp.
\]
These two parts are expected to behave very differently. On the one hand, the atypical part is conjectured to be algebraic:

\begin{conjecture}[{\cite[Conj. 2.5]{bku}}]
\label{zilberpink}
Let $\V$ be a polarizable $\Z$-VHS on a smooth, irreducible and quasi-projective variety $S$ over $\C$. The atypical Hodge locus $\HL_\atyp$ of $S$ for $\V$ is an algebraic suvariety of $S$.
\end{conjecture}

Indeed, Conjecture \ref{zilberpink} can be seen as an extension of the classical Zilber-Pink conjecture (see \cite[Conj. 2.3]{pila},\cite[Conj. 1.1]{pink} and \cite{zilber}). The case of positive period dimensional special subvarieties has been settled for variations with $\Q$-simple derived Mumford-Tate group in \cite[Thm. 3.1]{bku}. The case of atypical points appears for now to be entirely out of reach in full generality. Understanding the typical Hodge locus, however, appears more accessible, and towards this end Baldi-Klingler-Ullmo conjecture the following:

\begin{conjecture}[{\cite[Conj. 2.7]{bku}}]
\label{typdenseconj}
Let $\V$ be a polarizable $\Z$-VHS on a smooth, irreducible and quasi-projective variety $S$ over $\C$. If $\HL_\typ$ is non-empty, $\HL_\typ$ is analytically (hence Zariski) dense in $S$.
\end{conjecture}

Towards this conjecture they prove:

\begin{theorem}[{\cite[Thm. 3.9]{bku}}]
\label{typedensebku}
    Let $\V$ be a polarizable $\Z$-VHS on a smooth, irreducible and quasi-projective variety $S$ over $\C$. If the typical Hodge locus $\HL_\typ$ is non-empty, then $\HL$ is analytically (hence Zariski) dense in $S$.
\end{theorem}

Conjecture \ref{typdenseconj} and Theorem \ref{typedensebku} naturally raise the following question asked by Baldi-Klingler-Ullmo:

\begin{question}[{\cite[Quest. 2.9]{bku}}]
\label{criterionquestion}
Is there a simple criterion on the generic Hodge datum of a polarizable $\Z$-VHS to decide if its typical Hodge locus is non-empty?
\end{question}

Regarding this question, they introduce the following notion (see \cite[Sect. 4.6]{bku} for details):

\begin{definition}[{\cite[Def. 4.15]{bku}}]\label{level}
Let $\V$ be a polarizable $\Z$-VHS on a smooth, irreducible and quasi-projective variety $S$ over $\C$ with generic Hodge datum $(\G,D)$ and $\Q$-simple algebraic monodromy group $\mon$. Denote by $\h$ its Lie algebra, fix $x \in D$ a Hodge generic point and denote by
\[
\h \otimes_\Q \C = \bigoplus_{k \in \Z} \h_x^{-k,k}
\]
the induced Hodge decomposition. The \textit{level of $\V$} is the largest integer $k$ such that $\h_x^{-k,k} \neq \{0\}$. It doesn't depend on the choice of $x$.
\end{definition}
They almost completely answer Question \ref{criterionquestion} when the level of the $\Z$-VHS is at least $3$, by proving:

\begin{theorem}[{\cite[Thm. 3.3]{bku}}]
\label{levelcrit}
Let $\V$ be a polarizable $\Z$-VHS on a smooth, irreducible and quasi-projective variety $S$ over $\C$, with generic Hodge datum $(\G,D)$ and algebraic monodromy group $\mon$. Suppose that $\mon = \G^\der$. If $\V$ is of level at least $3$ then $\HL_\typ = \emptyset$ (and thus $\HL = \HL_\atyp$).
\end{theorem}

\subsection{Our Results} The main contribution of this work will be to provide a sufficient criterion for the density of (typical) Hodge loci, valid in any level. As we will explain, this almost gives a complete answer to Question \ref{criterionquestion}.

\subsubsection{Density of typical Hodge loci} Let $\V$ be a polarizable $\Z$-VHS on a smooth, irreducible and quasi-projective variety $S$ over $\C$ and let $(\G,D)$ be its generic Hodge datum. Note that after passing to a finite étale cover of the base (which doesn't affect our results, see Remark \ref{zariskiopen}), we can and will assume that the image of the monodromy representation of $\V$ is torsion-free. We therefore fix a torsion-free arithmetic lattice $\Gamma$ of $\G(\R)$ containing the latter and denote by $\Phi : S^\an \rightarrow \Gamma \quo D$ be the period map associated to $\V$. Given a strict Hodge sub-datum $(\M, D_M) \subsetneq (\G,D)$ we define the \textit{Hodge locus of type $\M$} as
\[
\HLM := \{s \in S^\an \hspace{0.1cm} : \hspace{0.1cm} \exists g \in \G(\Q)^+, \MT(\V_s) \subseteq g \M g^{-1} \},
\]
where for $s \in S^\an$, the Mumford-Tate group of the Hodge structure on $\V_s$ is denoted by $\MT(\V_s)$. The \textit{typical Hodge locus of type $\M$} is defined as the union $\HLM_\typ$ of those components of $\HLM$ which are typical with respect to their generic Hodge datum; see Definition \ref{typdef}.

\begin{definition}\label{expected}
Let $X,Y$ be locally closed irreducible analytic subvarieties of some irreducible complex analytic variety $Z$. Let $U$ be an analytic irreducible component of the intersection $X \cap Y$. We say that \textit{$X$ and $Y$ intersect in $Z$ with bigger dimension than expected along $U$} if
\[
\dim U > \dim X + \dim Y - \dim Z.
\]
Otherwise, we say that they \textit{intersect in $Z$ with expected dimension along $U$}.
\end{definition}

In this terminology, a special subvariety $Z$ of $S$ with generic Hodge datum $(\G_Z, D_{G_Z})$ is typical if it corresponds to a component of the intersection between $\Gamma_{\G_Z} \quo D_{G_Z}$ and $\Phi(S^\an)$ along which these two intersect with expected dimension inside $\Gamma \quo D$. Here $\Gamma_{\G_Z}$ denotes the arithmetic torsion-free lattice $\Gamma \cap \G_Z(\Q)^+$ of $\G_Z(\R)^+$. Then, a first trivial remark is that for $\HLM_\typ$ to be non-empty, the Hodge datum $(\M, D_M)$ has to be of the following type:

\begin{definition}
A strict Hodge sub-datum $(\M,D_M) \subsetneq (\G,D)$ is said \textit{$\V$-admissible} if it satisfies the inequality of dimensions
\[ \dim \Phi(S^\an) + \dim D_M - \dim D \geq 0. \]
We say that $(\M, D_M)$ is \textit{strongly $\V$-admissible} if the above inequality is strict.
\end{definition}

Our general result (see Theorem \ref{main}) is somewhat technical, but simplifies in the case where the algebraic monodromy group $\mon$ of $\mathbb{V}$ is $\mathbb{Q}$-simple and equal to $\G^\der$, saying that (strong) $\V$-admissibility of a sub-datum is enough to get density of the associated (typical) Hodge locus:

\begin{theorem}\label{mainqsimple}
Let $\V$ be a polarizable $\Z$-VHS on a smooth, irreducible and quasi-projective variety $S$ over $\C$ with generic Hodge datum $(\G,D)$. Assume that its algebraic monodromy group is $\mon = \G^\der$ and is $\Q$-simple. Let $(\M, D_M)\subsetneq (\G, D)$ be a strict Hodge sub-datum.
\begin{itemize}
\item[(i)] If $(\M, D_M)$ is $\V$-admissible, then $\HLM$ is analytically dense in $S^\an$. 
\item[(ii)] If $(\M, D_M)$ is strongly $\V$-admissible, then $\HLM_\typ$ is analytically dense in $S^\an$.
\end{itemize}
\end{theorem}

\begin{remark}
As pointed out in \cite[Rmk. 10.2]{bku}, Baldi-Klingler-Ullmo's geometric Zilber-Pink theorem gives (because we are in the $\Q$-simple monodromy case) that (i) implies (ii). We give an independent argument for (ii) that still works in the factor case.
\end{remark}

In terms of the full Hodge locus we thus have immediately:

\begin{corollary}\label{qsimplemon}
Let $\mathbb{V}$ be as in Theorem \ref{mainqsimple}. If $(\G,D)$ admits a (resp. strongly) $\V$-admissible strict Hodge sub-datum, the Hodge locus $\HL$ (resp. the typical Hodge locus $\HL_{\typ}$) of $S$ for $\V$ is analytically dense in $S^\an$. 
\end{corollary}

\begin{remark}\label{constfactorrem}
The condition that $\mon = \G^\der$ will be verified in most geometric applications (e.g. hypersurfaces variations, Hodge-generic subvarieties of $\ppabel_g$). It is however needed to exclude the eventuality that the datum $(\M,D_M)$ is $\V$-admissible but not big enough on the factor of $D$ on which $\Phi$ is constant to generically intersect the period image. The most general condition that one can hope for (see Theorem \ref{main}) is that, writing $\G^\der = \mon \cdot \LL$ as an almost direct product, $\M$ contains $\LL$. Otherwise, the statement would contradict the Zilber-Pink conjecture for an auxiliary $\Z$-VHS (see Example \ref{factorwisexample}).
\end{remark}

Working with not necessarily $\Q$-simple algebraic monodromy groups leads to complications. The following example illustrates the fact that in this general setting we cannot hope for Theorem \ref{mainqsimple} to remain true and that we need to impose similar numerical conditions on each factor.

\begin{example}\label{factorwisexample}
In this example, we fix a level structure $n \geqslant 3$ and for any $g \geqslant 1$ we denote by $\ppabel_g$ the quasi-projective Shimura variety of dimension $\frac{g(g+1)}{2}$ parametrising principally polarized abelian varieties of dimension $g$ with level-$n$-structure. Let $C \subset \ppabel_3$ be a Hodge generic curve. Let $\V$ be the $\Z$-VHS corresponding to the first cohomology groups of the members of the universal family of abelian six-folds above $S := \ppabel_3 \times C$ (recall that we fixed level structures so that we have fine moduli spaces). Its period map is the inclusion $\ppabel_3 \times C \hookrightarrow \ppabel_3 \times \ppabel_3$. Let $(\M,D_M)$ be a strict Hodge sub-datum with $\Gamma_\M \quo D_M = \ppabel_3 \times (\ppabel_1 \times \ppabel_2)$. It is strongly $\V$-admissible. If $\HLM_\typ$ was dense, one would find that $C \subset \mathcal{A}_{3}$ intersects Hecke translates of $\mathcal{A}_{1} \times \mathcal{A}_{2}$ in a dense subset. This contradicts the Zilber-Pink conjecture \ref{zilberpink} for the variation on $C$ corresponding to the first cohomology groups of the members of the universal family of abelian threefolds above $C$, which has period map $C \hookrightarrow \mathcal{A}_{3}$: $C$ is of dimension $1$ while $\ppabel_1 \times \ppabel_2$ has codimension $2$ in $\ppabel_3$, so the intersections of $C$ with Hecke translates of $\ppabel_1 \times \ppabel_2$ are atypical. So Theorem \ref{mainqsimple} should fail in the non-$\Q$-simple monodromy case.
\end{example}

Our criterion is nearly a complete answer to \ref{criterionquestion} in the sense that the only missing ingredient for our methods to give a full characterization of the density of typical Hodge loci is a better understanding of the distribution of typical special points in $S$. More precisely, using the methods of this work, we would need the following statement, which is a consequence of the Zilber-Pink conjecture \ref{zilberpink}, to get a full answer to \ref{criterionquestion}:

\begin{conjecture}\label{qsimpleconj}
Let $\V$ be as in Theorem \ref{mainqsimple}, and $(\M, D_M) \subsetneq (\G,D)$ be a strict Hodge sub-datum which satisfies
\[ \dim \Phi(S^\an) + \dim D_M = \dim D. \]
Then for any non-empty analytic open subset $B \subset S^\an$, one has a strict containment
\[ \HLM_\atyp \cap B \subsetneq \HLM \cap B . \]
\end{conjecture}

Indeed, the reason for distinguishing between admissibility and strong admissibility in Theorem \ref{mainqsimple}(ii) ultimately arises from the inability to handle typicality questions in the situation where one expects $\HLM$ to consist of typical special points. 

\begin{remark}
Combining recent work of Tayou-Tholozan and the density of $\HL_\typ$ in $S$ that we prove here, one should be able to obtain that this locus is even equidistributed in $S$ for a natural measure. We refer the interested reader to their paper \cite{TT}.
\end{remark}

\subsubsection{Classical Noether-Lefschetz loci} 

We now explain how to apply our results to reprove classical theorems on the density of Noether-Lefschetz loci. In particular, we will show how our results generalize the classical results of Ciliberto-Harris-Miranda \cite{zbMATH04097545} and Green \cite[Prop. 17.20]{voi} by proving the analytic density of the components of maximal codimension of the Noether-Lefschetz locus of degree $d$ hypersurfaces in $\mathbb{P}^{3}$ as soon as $d \geq 5$. (The case $d = 4$ requires extra analysis, which we omit.)

In this case one considers the variation $\mathbb{V} = (R^{2} f_{*} \underline{\mathbb{Z}})_{\textrm{prim}}$ on primitive cohomology associated to the universal family $f$ of such hypersurfaces, and the polarization $Q : \mathbb{V} \otimes \mathbb{V} \to \underline{\mathbb{Z}}$ induced by cup-product. Fix some $s \in S^\an$. The result of Beauville \cite{zbMATH03973031} guarantees that the algebraic monodromy group $\mon$ of $\mathbb{V}$ is the full orthogonal group $\textrm{Aut}(\mathbb{V}_{s}, Q_{s})$ stabilizing the polarizing form. As $Q_{s}$ is a Hodge class in $(\mathbb{V}_{s} \otimes \mathbb{V}_{s})^{\vee}$ and the algebraic monodromy group lies inside the generic Mumford-Tate group, this means that the generic Hodge datum for $\mathbb{V}$ takes the form $(\G, D)$, with $\G = \textrm{GAut}(\mathbb{V}_{s}, Q_{s})$ and with $D$ the space of all polarized Hodge structures with the same Hodge numbers $(h^{2,0}, h^{1,1}, h^{0,2})$ as $\mathbb{V}$. Given any subdatum $(\M, D_{M}) \subset (\G, D)$ corresponding to a single fixed Hodge vector, one computes that $\dim D - \dim D_{M} = h^{2,0}$. On the other hand it is classical that $h^{2,0} = {d-1 \choose 3}$, hence one has for $d \geq 5$ that 
\[ \dim D - \dim D_{M} = {d-1 \choose 3} < {d+3 \choose 3} - 16 = \dim \Phi(S^\an) . \]
Here we have computed the period dimension by observing that the projective space of degree $d$ homogeneous polynomials in $4$ variables is ${d+3 \choose 3} - 1$ dimensional, and the period map becomes generically injective modulo the natural action of $\textrm{SL}_{4}$ on the moduli space as a consequence of Donagi's generic Torelli theorem \cite{zbMATH03964006}, which holds assuming $d \geq 5$. Thus the datum $(\M, D_M)$ is strongly $\mathbb{V}$-admissible, and Theorem \ref{mainqsimple} lets us conclude.

We expect that the above argument, which requires essentially only a verification of Zariski dense monodromy and a generic Torelli theorem, can be repeated to reprove numerous other analytic density results in the literature.

\subsubsection{Special subvarieties of the Torelli locus}

Let $g \geqslant 4$ be an integer, and fix a level structure $n \geqslant 3$. Let $\ppabel_g$ be the Shimura variety parametrising principally polarized abelian varieties of dimension $g$ (with the given level structure) whose associated Shimura datum is $(\gsp_{2g}, \halfplane_g)$, where $\halfplane_g$ is the Siegel upper half-plane. Recall that $\ppabel_g$ is a quasi-projective complex variety of dimension $\frac{g(g+1)}{2}$ and that the largest dimension obtained by any of its strict special subvarieties is $\frac{g(g-1)}{2} + 1$, realised for example by $\ppabel_{g-1} \times \ppabel_1$. As suggested in \cite[Rmk 3.15]{bku}, one expects in this setting that:

\begin{conjecture}
Let $S \subset \ppabel_g$ be a closed Hodge generic subvariety of dimension $q$, and $\V$ the induced polarizable $\Z$-VHS on $S$. The typical Hodge locus of $S$ for $\V$ is analytically dense if and only if $q \geqslant g-1$. In that case, the Hecke translates of $\ppabel_{g-1} \times \ppabel_{1}$ in $\ppabel_g$ intersect $S$ in an analytically dense subset.
\end{conjecture}

Unconditionally, Theorem \ref{mainqsimple} gives the following:

\begin{corollary}\label{Ag}
Let $S \subset \ppabel_g$ be a closed Hodge generic subvariety of dimension $q$, and $\V$ the induced polarizable $\Z$-VHS on $S$. If $q \geqslant g$, then the typical Hodge locus of $S$ for $\V$ is analytically dense. In that case, the Hecke translates of $\ppabel_{g-1} \times \ppabel_{1}$ in $\ppabel_g$ intersect $S$ in an analytically (positive dimensional) dense subset.
\end{corollary}

\begin{proof}
$\V$ has generic Mumford-Tate group $\G = \gsp_{2g}$ whose derived subgroup $\G^\der = \ssp_{2g}$ is $\Q$-simple. In particular, since $\V$ is non-constant and the algebraic monodromy group $\mon$ is a $\mathbb{Q}$-normal subgroup of the derived Mumford-Tate group, one has $\mon = \G^\der = \ssp_{2g}$. To conclude, one simply remarks that the inequality $q \geqslant g$ is equivalent to the strong $\V$-admissiblity of the Hodge sub-datum corresponding to $\mathcal{A}_{g-1} \times \mathcal{A}_{1}$.
\end{proof}

Let $\cur_g$ be the moduli space of smooth projective curves of genus $g$ with level-$n$-structure, let $j : \cur_g \hookrightarrow \ppabel_g$ the Torelli morphism, 
and $\torelli^0_g = j(\cur_g) \subset \ppabel_g$ the open Torelli locus. In this setting, Corollary \ref{Ag} gives:

\begin{corollary}\label{torelli}
The typical Hodge locus of $\cur_g$ for the polarizable $\Z$-VHS induced by $j$ is analytically dense in $\cur_g$.
\end{corollary}

\begin{proof}
$\torelli^0_g$ is known to be of dimension $3g - 3$ which is bigger than $g$ in our setting ($g \geqslant 4$). Therefore, applying Corollary \ref{Ag} to the closure of $\torelli^0_g$ in $\ppabel_g$ gives the density of the Hodge locus in $\cur_g$ for the analytic topology.
\end{proof}

\begin{remark}
Similar questions had already been studied in \cite{izadi, colombopirola} for $\ppabel_g$ and \cite{chai} in the more general setting of Shimura varieties. Chai had developed a numerology to give a sufficient condition on the codimension of a subvariety of a Shimura variety to have analytically dense Hodge locus of given type $\M$. 
These conditions weren't however optimal, and Theorem \ref{mainqsimple} (and its generalisation Theorem \ref{main}) can be seen as a refinement of these works in three directions. First, we give numerical conditions in a general Hodge theoretic setting while previous authors only worked in the Shimura setting. Furthermore, our numerical conditions are sharper than previous ones, even in the Shimura setting. For example, they only knew that subvarieties of $\ppabel_g$ had analytically dense Hodge locus if they had codimension less than $g$. We prove the same thing for subvarieties of dimension greater than $g$. Finally, the notion of typicality doesn't appear in the aforementionned works, so proving density of the typical part of the Hodge locus is an additional improvement.
\end{remark}

\subsection{Recent related work}

Recently, S. Eterovi{\'c} and T. Scanlon have obtained similar results in a general framework of so-called definable arithmetic quotients. In particular, they also give \cite[Theorem 3.4]{es} a criterion for Hodge loci to be dense. Our work can be seen as refining these results in the Hodge-theoretic setting, taking into account the Hodge-theoretic (a)typicality of the resulting loci. More precisely, we prove in Theorem \ref{mainqsimple}(ii) and Theorem \ref{main}(ii) that the components of $\HLM$ that we construct have the expected Mumford-Tate groups, whereas a conclusion of this type does not appear in the main theorem of \cite{es}.

\subsection{Aknowledgements}

We thank Emmanuel Ullmo and Gregorio Baldi for helpful feedback on a draft of this manuscript and for many useful conversations during the completion of this work. The second author would also like to thank Thomas Scanlon for explaining to him some aspects of his manuscript \cite{es}. Finally we thank the anonymous referee whose comments helped improve the presentation and corrected some inaccuracies.

\section{Recollections}

We now review notions from variational Hodge theory; we refer to \cite{voi}, \cite{ggk}, \cite{klingler_atyp} and \cite{klingicm} for a detailed account of the theory. The first two subsections recall the basic notions, and the subsequent two subsections will review more recent results which we will need for our proofs.
Before starting, let us emphasize on the following:

\begin{remark}\label{zariskiopen}
In this work, we are only interested on proving analytic density results for Hodge loci in the base $S$ of some $\Z$-VHS. In particular, to prove them, we are free to replace the base $S$ by:
\begin{itemize}
\item[(i)] some non-empty Zariski open subset, as $S$ being irreducible, such a subset will be analytically dense in $S$;
\item[(ii)] some finite étale covering, as the image of an analytically dense subset of the cover will remain such in $S$.
\end{itemize}
We will do such substitutions freely in the sequel.
\end{remark}

\begin{notation}
For an algebraic group $\G$ over $\Q$, we denote by $\G^\ad$ its adjoint group, by $\G^\der$ its derived subgroup, and by $\G(\R)^+$ the identity component for the real analytic topology. Finally, we define $\G(\Q)^+$ to be $\G(\R)^+ \cap \G(\Q)$.
\end{notation}

\subsection{Hodge data and special subvarieties}

Recall that the Mumford-Tate group of a $\Q$-Hodge structure $x : \s \rightarrow \gl(V_\R)$ on a finite dimensional $\Q$-vector space $V$ (where $\s = \res_{\C/\R} \mult_m$ is the Deligne torus) is the smallest $\Q$-subgroup $\MT(x)$ of $\gl(V)$ such that the morphism of $\R$-algebraic groups $x$ factors through $\MT(x)_\R$. Equivalently, it is the fixator in $\gl(V)$ of all Hodge tensors for $x$. It is a connected $\Q$-algebraic group, which is reductive as soon as $x$ is polarized, which will always be assumed in the sequel. The associated Mumford-Tate domain is the orbit of $x$ under the identity connected component $\MT(x)(\R)^+$ of the real Lie group $\MT(x)(\R)$ in $\Hom(\s, \MT(x)_\R)$. Both notions can be summarized in the following, now usual, definition:

\begin{definition}[{\cite[Def. 3.4, Def. 3.5]{klingler_atyp}}] \label{hodgedata}
~
\begin{itemize}
\item[(i)] A \textit{Hodge datum} is a pair $(\G,D)$ where $\G$ is the Mumford-Tate group of some Hodge structure, and $D$ is the associated Mumford-Tate domain. 
\item[(ii)] A \textit{morphism of Hodge data} $(\G,D) \rightarrow (\G',D')$ is a morphism of $\Q$-groups $\G \rightarrow \G'$ sending $D$ inside $D'$. 
\item[(iii)] A \textit{Hodge sub-datum} of $(\G, D)$ is a Hodge datum $(\G', D')$ such that $\G'$ is a $\Q$-subgroup of $\G$ and the inclusion $\G' \hookrightarrow \G$ induces a morphism of Hodge data $(\G',D') \rightarrow (\G,D)$.
\item[(iv)] A \textit{Hodge variety} is a quotient variety of the form $\Gamma \quo D$ for some Hodge datum $(\G,D)$ and some torsion-free arithmetic lattice $\Gamma \subset \G(\Q)^+$.
\end{itemize}
\end{definition}

\begin{remark}
Note that what we call here a \textit{Hodge datum} is usually referred to as a \textit{connected Hodge datum} in the literature. Since we will always be considering connected Hodge data, we ommit the "connected" in the sequel.
\end{remark}

Let $\V$ be a polarizable $\Z$-VHS on a smooth, irreducible and quasi-projective algebraic variety $S$ over $\C$. There exists a reductive $\Q$-algebraic group $\G$ and a countable union $\HL$ of irreducible algebraic subvarieties of $S$ such that for each point $s \in S^\an$, the polarized Hodge structure carried by $\V_s$ has Mumford-Tate group isomorphic (by parallel transport) to a subgroup of $\G$ which is strictly contained in $\G$ if and only if $s \in \HL$ (\cite{cdk}, \cite{bkt}). The group $\G$ is called the \textit{generic Mumford-Tate group of $\V$}, the locus $\HL$ is called the \textit{Hodge locus of $S$ for $\V$}, and points $s \in S^\an - \HL$ are said to be \textit{Hodge generic}. The orbit in $\Hom(\s, \G_\R)$ of the Hodge structure carried by $\V_s$ under $\G(\R)^+$ doesn't depend on the choice of $s \in S^\an$ and is denoted by $D$. We will refer to the Hodge datum $(\G,D)$ as the \textit{generic Hodge datum of $\V$}. Then, up to replacing $S^\an$ by a finite étale cover, one can associate to $\V$ a holomorphic, locally liftable to $D$ and horizontal map
\[
\Phi : S^\an \rightarrow \Gamma \quo D
\]
where $\Gamma \subset \G(\Q)^+$ is a torsion-free arithmetic lattice containing the image of the monodromy representation associated to $\V$. We will always refer to this map as \textit{the period map associated to $\V$} in the sequel.

If $Y$ is an irreducible algebraic subvariety of $S$, we can apply the previous paragraph to $\V \vert_{Y^{sm}}$ (where $Y^{sm}$ denotes the smooth locus of $Y$) to define the \textit{generic Hodge datum $(\G_Y, D_{G_Y})$ of $Y$ for $\V$}. We see it as a Hodge sub-datum of the generic Hodge datum $(\G,D)$ of $\V$.

\begin{definition}[{\cite[Def. 1.2]{ko}}]\label{caracspec}
An irreducible algebraic subvariety $Y \subset S$ is called a \textit{special subvariety of $S$ for $\V$} if it is maximal for the inclusion among irreducible algebraic subvarieties of $S$ whose generic Mumford-Tate group is $\G_Y$.
\end{definition}

Special subvarieties of $S$ for $\V$ can be characterized in more geometric terms.
A \textit{special subvariety of the Hodge variety $\Gamma \quo D$} is defined (see \cite[Def. 3.12]{klingler_atyp}) as the image of a map
\[
[f] : \Gamma' \quo D' \rightarrow \Gamma \quo D 
\]
between Hodge varieties induced by a morphism of Hodge data $f : (\G',D') \rightarrow (\G,D)$, where $\Gamma'\subset \G'(\Q)^+$ is a torsion-free arithmetic lattice such that $f(\Gamma') \subset \Gamma$ (such a map will be called, following \cite[Lem. 3.9.]{klingler_atyp}, a \textit{morphism of Hodge varieties}). Then we have the:

\begin{proposition}[{\cite[Prop 3.20]{klingler_atyp}}]
An irreducible algebraic subvariety $Y \subset S$ is special for $\V$ if and only if $Y$ is an irreducible component of the preimage by $\Phi$ of a special subvariety of the Hodge variety $\Gamma \quo D$.
\end{proposition}

This characterisation, which realises special subvarieties as preimages by the period map $\Phi$ of intersections between $\Phi(S^\an)$ and some analytic subvarieties of $\Gamma \quo D$, naturally suggests the dichotomy presented in the introduction between typical and atypical special subvarieties. More precisely, one defines:

\begin{definition} \label{typdef}
Let $Y$ be a special subvariety of $S$ for $\V$ with generic Hodge datum $(\G_Y, D_{\G_Y})$. It is said to be \textit{atypical} if $\Phi(S^\an)$ and $\Gamma_{\G_Y} \quo D_{G_Y}$ intersect with bigger dimension than expected along $\Phi(Y^\an)$ inside of $\Gamma \quo D$, i.e.
\[
\codim_{\Gamma \quo D} \Phi(Y^\an) < \codim_{\Gamma \quo D}\Phi(S^\an) + \codim_{\Gamma \quo D}\Gamma_{\G_Y} \quo D_{G_Y}
\]
Otherwise, it is said to be \textit{typical}. 
\end{definition}

\begin{remark}
Note the difference with \cite[Def. 2.2]{bku}: here we allow typical special subvarieties to be singular for $\V$ (i.e. to have image by $\Phi$ fully contained in the singular locus of $\Phi(S^\an)$, see \textit{op. cit.}). Indeed, because our results are insensitive to replacing the base by some non-empty Zariski open subset of it, we are free to remove the singular locus of $\Phi(S^\an)$ (see \cite[Rmk. 2.3.]{bku} or Step $1$ of Section \ref{pfmain}). Thus, we choose to use this simpler definition. Note that this definition also differs from Klingler's original one (\cite[Def. 4.3]{klingler_atyp}), which allows typical subvarieties to lie in the singular locus of $\V$ but uses codimensions inside of the horizontal distribution of $\Gamma \quo D$.
\end{remark}

Then, the Hodge locus splits in two parts 
\[
\HL = \HL_\typ \cup \HL_\atyp
\]
where $\HL_\typ$ (resp. $\HL_\atyp$) is defined as the union of the strict typical (resp. atypical) special subvarieties of $S$ for $\V$, and is referred to as the \textit{typical (resp. atypical) Hodge locus}.

\subsection{Monodromy and weakly special subvarieties}\label{monowspsec}

Let $\V$ be a $\Z$-VHS on a smooth, irreducible and quasi-projective algebraic variety $S$ over $\C$, let $(\G,D)$ be its generic Hodge datum, and fix $s_0 \in S^\an$ a Hodge generic point such that the Mumford-Tate group of the Hodge structure carried by $\V_{s_0}$ is equal (and not only isomorphic) to $\G$. The local system $\V_\Q := \V \otimes_\Z \Q$ corresponds to a representation of the topological fundamental group of $S^\an$ with base-point $s_0$:
\[
\rho : \pi_1(S^\an,s_0) \rightarrow \gl(\V_{\Q,s_0})
\]
We define the \textit{algebraic monodromy group $\mon$ of $\V$} as the identity connected component of the $\Q$-Zariski closure of $\rho(\pi_1(S^\an, s_0))$ in $\gl(\V_{\Q,s_0})$. More generally, if $Y \subset S$ is an irreducible algebraic subvariety, we define the \textit{algebraic monodromy group $\mon_Y$ of $Y$ for $\V$} as the algebraic monodromy group of the restricted $\Z$-VHS $\V \vert_{Y^\sm}$ on the smooth locus $Y^\sm$ of $Y$. Then, recall the following celebrated result of André which follows from the theorem of the fixed part and the semi-simplicity theorem of Deligne \cite{hodge2} (in the geometric setting) and Schmid \cite{schmid} (for general $\Z$-VHS).

\begin{theorem}[{\cite[Thm. 1]{andre}}]\label{andre}
Let $Y \subset S$ be an irreducible algebraic subvariety. Then the monodromy group $\mon_Y$ of $Y$ for $\V$ is a normal subgroup of the derived subgroup $\G_Y^\der$ of the generic Mumford-Tate group of $Y$ for $\V$.
\end{theorem}

In analogy with Proposition \ref{caracspec} one defines:

\begin{definition}
An irreducible algebraic subvariety $Y \subset S$ is said \textit{weakly special for $\V$} if it is maximal for the inclusion among irreducible algebraic subvarieties of $S$ which have algebraic monodromy group $\mon_Y$.
\end{definition}

As in the case of special subvarieties, there is a nice group theoretic description of weakly special subvarieties. A \textit{weakly special subvariety of the Hodge variety $\Gamma \quo D$} is defined (see. \cite[Def. 4.1]{ko}) to be the image of a subvariety of the form
\[
\Gamma_\M \quo D_M \times \{t\} \subset \Gamma_\M \quo D_M \times \Gamma_\NN \quo D_N
\]
by a morphism of Hodge varieties $\Gamma_\M \quo D_M \times \Gamma_\NN \quo D_N \rightarrow \Gamma \quo D$, with $t \in \Gamma_\NN \quo D_N$. Now we have the following:

\begin{proposition}[{\cite[Cor. 4.14]{ko}}]\label{caracweakly}
Let $Y \subset S$ be an irreducible algebraic subvariety. $Y$ is a weakly special subvariety of $S$ for $\V$ if and only if it is the preimage by $\Phi$ of a weakly special subvariety of $\Gamma \quo D$.
\end{proposition}

\begin{remark}\label{monperiodmap}
Note that by André's Theorem \ref{andre}, the derived subgroup of $\G$ factors as $\G^\der = \mon \cdot \LL$, and one shows \cite[Lem. 4.12]{ko} that the period map $\Phi$ is constant equal to some Hodge generic $t_L \in \Gamma_\LL \quo D_L$ when projected to the factor corresponding to $\LL$. In the sequel, we will omit this constant factor and simply write $\Phi : S^\an \rightarrow \Gamma_\mon \quo D_H$ for the period map to signify that up to passing to a finite étale cover we have $\Phi : S^\an \rightarrow \Gamma_\mon \quo D_H \times \{t_L\} \subset \Gamma \quo D$.
\end{remark}

\subsection{Algebraicity of period maps} 

We will make use of the following result on algebraicity of period maps, recently proven by Bakker-Brunebarbe-Tsimerman.

\begin{theorem}[{\cite[Thm. 1.1]{bbt}}] \label{algpermap}
Let $\V$ be a polarizable $\Z$-VHS on a smooth and quasi-projective algebraic variety $S$ over $\C$, and denote by $\Phi : S^\an \rightarrow \Gamma \quo D$ the associated period map. The map $\Phi$ factors uniquely as $\Phi = \iota \circ f^\an$ where $f : S \rightarrow T$ is a dominant regular map to an algebraic variety $T$, and $\iota : T^\an \hookrightarrow \Gamma \quo D$ is a closed analytic immersion.
\end{theorem}
\subsection{Functional transcendance}
Our arguments heavily rely on a functional transcendence result, Bakker and Tsimerman's Ax-Schanuel theorem, which will allow us to force some intersections to have expected dimension. Recall that the Mumford-Tate domain $D_H$ embeds in its compact dual $\check{D}_H$ which is a projective algebraic variety over $\C$. This embedding allows one to define an irreducible algebraic subvariety of $S \times D_H$ as an analytic irreducible component of the intersection $(S \times D_H) \cap V$ of $S \times D_H$ with an algebraic subvariety $V$ of $S \times \check{D}_H$. Bakker-Tsimerman prove:

\begin{theorem}[{\cite[Thm. 1.1]{baktsi}}] \label{axschan}
Let $W \subset S \times D_H$ be an algebraic subvariety. Let $U$ be an irreducible complex analytic component of $W \cap (S \times_{\Gamma_\mon \quo D_H} D_H)$ such that
\[
\codim_{S \times D_H} U < \codim_{S \times D_H} W + \codim_{S \times D_H} \big(S \times_{\Gamma_\mon \quo D_H} D_H\big).
\]
Then the projection of $U$ to $S$ is contained in a strict weakly special subvariety of $S$ for $\V$.
\end{theorem}

\begin{remark}
\label{ominrem}
The proof of Bakker-Tsimerman's Ax-Schanuel crucially relies on the definability of period maps in the o-minimal structure $\R_{\an,\exp}$. It has been recently reproven in a more general setting using an o-minimal free approach (\cite{diffaxschan}).
\end{remark}

\section{Density of typical special subvarieties}
Let $\V$ be a polarizable $\Z$-VHS on a smooth, irreducible and quasi-projective variety $S$ over $\C$. Denote by $(\G,D)$ the generic Hodge datum, by $\Gamma \subset \G(\Q)^+$ some torsion-free arithmetic lattice containing the image of the monodromy representation of $\V$, by $\mon$ the algebraic monodromy group and by $\Phi : S^\an \rightarrow \Gamma_\mon \quo D_H \subset \Gamma \quo D$ the associated period map (see Remark \ref{monperiodmap}).

In the sequel, definability will always be meant in the o-minimal structure $\anexp$. We use the notion of definable manifold appearing in \cite[Def. 2.1, Def. 2.2]{bbt}. The natural Borel embedding $D \hookrightarrow \check{D}$ as an open subset of the compact dual defines a canonical structure of $\anexp$-definable (and even $\alg$-definable) manifold on $D$. Similarly, the algebraic structure on $S$ induces a canonical structure of $\anexp$-definable (and again even $\alg$-definable) manifold on $S^\an$. Definability on $D$ and $S^\an$ will always be meant with respect to these structures.

\subsection{Definable fundamental sets}

In the proofs and statements of intermediate lemmas, we will need to work with definable fundamental sets for the uniformization maps that appear, which are compatible with period maps. However, we won't need the delicate definable factorization in \cite{bkt} of period maps through fundamental sets for the $\Gamma$-action on $D$, but only, following \cite[Sect. 3]{baktsi} and \cite[Sect. 6.1.1]{bku}, definable fundamental sets for the covering action of the fundamental group of $S$ on a suitable cover of $S$, to which the restriction of the lifted period map is definable. We recall quickly the construction for completeness.

Let $(\Bar{S}, E)$ be a log-smooth compactification of $S$. By the Borel monodromy theorem \cite[(4.5)]{schmid}, and because our results are unaffected by passing to a finite étale cover of $S$ (see Remark \ref{zariskiopen}), we can and do assume that $\V$ has unipotent monodromy along irreducible branches of $E$. The datum of this compactification allows us to choose a definable atlas of $S$ for the definable structure induced by its algebraic structure, given by finitely many (by compactness of $\Bar{S}$) polydisks $P_i = \Delta^{k_i} \times (\Delta^*)^{l_i}$ (because the divisor $E$ has simple normal crossings). Let $p_i = \exp : \tilde{P}_i := \Delta^{k_i} \times \poinca^{l_i} \rightarrow P_i$ be the universal cover and pick some $b, c > 0$ such that, denoting by $\Sigma = (-b,b) \times (c, +\infty[ \subseteq \poinca$ the associated vertical strip, the set $\F_i := \Delta^{k_i} \times (\Sigma)^{l_i}$ is a fundamental set for the covering action of $\pi_1(P_i) \cong \Z^{l_i}$. We endow $\F_i$ with the natural product definable structure, with on $\poinca$ the structure coming from the natural embedding $\poinca \hookrightarrow \p^1$. With these structures, the covering maps $p_i$ are definable when restricted to $\mathcal{F}_{i}$.

For each $i$, choose a local lifting $\tilde{\Phi}_i$ of $\Phi \vert_{P_i}$, and letting $\F$ be the disjoint union of these local fundamental sets over the charts, view the collection of all such lifts as a map on $\F$. Then, because we assumed unipotent monodromy at infinity, a straightforward application of the Nilpotent Orbit Theorem \cite[(4.12)]{schmid} (see \cite[Lem. 3.1]{baktsi}) shows that, after possibly shrinking the initial charts, we get the following diagram in the category of definable complex analytic manifolds:
\begin{equation}
\label{Idefdiag}
\xymatrix{\F \ar[r]^{\tilde{\Phi}_\F} \ar[d]_\exp & D \\ S^\an}
\end{equation}
\begin{notation}
In the sequel, we fix once and for all the data needed to construct this diagram, and will denote $\I = \tilde{\Phi}_\F(\F)$. It is not a priori connected but, although it won't be useful in the sequel, it is possible to choose the local lifts in the construction so that it is so. In particular, if two local lifts $\tilde{\Phi}_{i_1}$ and $\tilde{\Phi}_{i_2}$ have intersecting images in $\Gamma \backslash D$, then their images in $D$ also intersect after replacing $\tilde{\Phi}_{i_1}$ by $\gamma \cdot \tilde{\Phi}_{i_1}$ for some $\gamma \in \Gamma$. Since the image of $\I$ in $\Gamma \backslash D$ is just the (necessarily connected) image of $S$, we can therefore ensure $\I$ is connected after finitely many such adjustments.
\end{notation}

\subsection{Statement of the result in the general setting}

In order to tackle product phenomena we work, as in \cite[Sect. 1.2]{bku}, in the setting of \cite[Chap. III]{ggk} which we recall now for completeness. The monodromy group $\mon$ decomposes as an almost direct product of its $\Q$-simple factors:
\[
\mon = \G_1  \cdots \G_n.
\]
After replacing $S$ by a finite étale cover (see Remark \ref{zariskiopen}), the above decomposition induces a factorization of the period map
\[
\Phi : S^\an \rightarrow \Gamma_1 \quo D_1 \times \cdots \times \Gamma_n \quo D_n,
\]
where for each $i$, $D_i$ is some Mumford-Tate subdomain of $D$ associated to $\G_i$ and $\Gamma_i = \Gamma \cap \G_i(\Q)^+$. For each $I \subset \{1, \cdots, n\}$, we write $D_I := \prod_{i \in I} D_i$ which is a Mumford-Tate domain for the group $\G_I := \prod_{i \in I} \G_i$ endowed with a projection map $p_I : D \rightarrow D_I$. We also write
\[ q_I : \Gamma \quo D \rightarrow \Gamma_I \quo D_I := \prod_{i \in I} \Gamma_i \quo D_i\]
for the natural projection and $\Phi_I := q_I \circ \Phi$ for the projected period map.

\vspace{0.5em}

\begin{remark}\label{splitaction}
The action of $\mon(\R)^+$ on $D_1 \times \cdots \times D_n$ is compatible with the product decomposition. Namely, for $h \in \mon(\R)^+$, we can write $h = g_1 \cdots g_n$ with $g_i \in \G_i(\R)$ and then, for $x = (x_1, \cdots, x_n)$, one has:
\[ h \cdot x = (g_1 \cdot x_1, \cdots, g_n \cdot x_n).\]
This applies for any choice of decomposition $h = g_1 \cdots g_n$. The same applies more generally for the decomposition $D = D_{H} \times D_{L}$ associated to the almost-direct factorization $\mathbf{G}^{\textrm{der}} = \mathbf{H} \cdot \mathbf{L}$.
\end{remark}

\vspace{0.5em}

When trying to give a sufficient condition generalizing that of Theorem \ref{mainqsimple} to the not necessarily $\Q$-simple monodromy case, Example \ref{factorwisexample} suggests that one should at least impose conditions like $\V$-admissibility on each factor, including on the factor where the period map is constant.

\begin{definition}\label{nf}
Let $(\M, D_M) \subsetneq (\G, D)$ be a strict Hodge sub-datum. It is called
\begin{itemize}
\item[(i)] \textit{factorwise $\V$-admissible} if for every non-empty set of indexes $I \subseteq \{1, \cdots, n\}$ the inequality
\[
\dim p_I(D_M) + \dim \Phi_I(S^\an) - \dim D_I \geq 0
\]
holds.
\item[(ii)] \textit{factorwise strongly $\V$-admissible} if for every non-empty set of indexes $I \subseteq \{1, \cdots, n\}$, the strict inequality
\[
\dim p_I(D_M) + \dim \Phi_I(S^\an) - \dim D_I > 0
\]
holds.
\end{itemize}
\end{definition}

As explained in Remark \ref{constfactorrem}, one should also add some condition on the constant factor:

\begin{definition}
A strict Hodge sub-datum $(\M,D_M) \subsetneq (\G,D)$ is said to be \textit{full on the $\LL$-factor} if, writing $\G^\der = \mon \cdot \LL$ as an almost-direct product, the normal subgoup $\LL$ of $\G^\der$ is contained in $\M$.
\end{definition}

We can now state the main result of this paper:

\begin{theorem}\label{main}
Let $\V$ be a polarizable $\Z$-VHS on a smooth, irreducible and quasi-projective variety $S$ over $\C$, with generic Hodge datum $(\G,D)$. Let $(\M, D_M) \subsetneq (\G,D)$ be a strict Hodge sub-datum.
\begin{itemize}
    \item[(i)] If $(\M,D_M)$ is full on the $\LL$-factor and factorwise $\V$-admissible, $\HLM$ is analytically dense in $S^\an$.
    \item[(ii)] If $\mon = \G^\der$ and $(\M,D_M)$ is factorwise strongly $\V$-admissible, $\HLM_\typ$ is analytically dense in $S^\an$.
\end{itemize}
\end{theorem}

\subsection{Reformulation of factorwise $\V$-admissibility}

To prove this result, we will need to reformulate the factorwise $\V$-admissibility condition in a way which is better suited to our proof, although more complicated to define. 
First define:

\begin{definition}
Let $(\M,D_M) \subsetneq (\G,D)$ be a strict Hodge sub-datum, and $I \subseteq \{1, \cdots, n\}$ be a non-empty set of indexes whose complementary set we denote by $I^c$. A pair $(g,t) \in \G(\R)^+ \times (D_{I^c} \times D_L)$ is called a \textit{$(\I, D_M, D_I)$-intersecting pair} if
\[ 
\I \cap (g \cdot D_M) \cap (D_I \times \{t\}) \neq \emptyset.
\]
\end{definition}

We will need:

\begin{lemma}\label{indepcap}
Let $(\M,D_M) \subsetneq (\G,D)$ be a strict Hodge sub-datum, let $I \subseteq \{1, \cdots, n\}$ be a non-empty set of indexes and $(g,t) \in \G(\R)^+ \times (D_{I^c} \times D_L)$ be some $(\I, D_M, D_I)$-intersecting pair. Up to replacing $S$ by some non-empty Zariski open subset, the quantity
\[
d_I(\M, D_M) := \dim\big((g \cdot D_M) \cap (D_{I} \times \{t\})\big) + \dim\big(\I \cap (D_{I} \times \{t\})\big) - \dim D_I
\]
only depends on the set of indexes $I$ and not on the choice of the $(\I, D_M, D_I)$-intersecting pair $(g,t)$.
\end{lemma}

\begin{remark}\label{modif}
In Lemma \ref{indepcap}, we emphasize that by ``replacing $S$ by some non-empty Zariski open subset'' we mean that in the diagram (\ref{Idefdiag}), we replace $S$ by some non-empty Zariski open subset, the fundamental domain $\F$ by the preimage of this open and $\tilde{\Phi}_\F$ by its restriction to this preimage. In particular, the image $\I$ is replaced by the image of the above restriction of $\tilde{\Phi}_\F$.
\end{remark}

\begin{convention}\label{conv}
Let $(\M,D_M) \subsetneq (\G,D)$ be a strict Hodge sub-datum. By convention, we set $d_\emptyset(\M,D_M) = 0$. 
\end{convention}

We then have the following reformulation, where we write $d_\mon(\M,D_M)$ for $d_{\{1, \cdots,n\}}(\M,D_M)$:

\begin{lemma}\label{reform}
Let $(\M,D_M)$ be a strict Hodge sub-datum. It is:
\begin{itemize}
    \item[(i)] factorwise $\V$-admissible if and only if for every strict set of indexes $I \subsetneq \{1, \cdots, n\}$, the inequality
    \[
    d_\mon(\M,D_M) \geq d_I(\M,D_M)
    \]
    holds.
    \item[(ii)] factorwise strongly $\V$-admissible if and only if for every strict set of indexes $I \subsetneq \{1, \cdots, n\}$, the strict inequality
    \[
    d_\mon(\M,D_M) > d_I(\M,D_M)
    \]
    holds.
\end{itemize}
\end{lemma}

\section{Proofs}
\subsection{Proof of lemmas \ref{indepcap} and \ref{reform}}
\begin{proof}[Proof of Lemma \ref{indepcap}]
Let us first prove that for any other choice of $(\I, D_M, D_I)$-intersecting pair $(g',t')$, we have:
\[
\dim\big((g \cdot D_M) \cap (D_{I} \times \{t\})\big) = \dim\big((g' \cdot D_M) \cap (D_{I} \times \{t'\})\big).
\]
Take $x \in (g \cdot D_M) \cap (D_{I} \times \{t\})$ and $x' \in (g' \cdot D_M) \cap (D_{I} \times \{t'\})$. Since $g' M g'^{-1}$ acts transitively on $g' \cdot D_M$ (where we set $M = \M(\R)^+$), there is some $m \in g' M g'^{-1}$ sending $(g'g^{-1})x$ to $x'$. We then have
\[
x' \in (mg'g^{-1})((g \cdot D_M) \cap (D_{I} \times \{t\})) = (g' \cdot D_M) \cap (mg'g^{-1})(D_{I} \times \{t\})
\]
and necessarily, $(mg'g^{-1})(D_{I} \times \{t\}) = D_I \times \{t'\}$. Indeed, the element $h := mg'g^{-1} \in \G(\R)^+$ writes $h' \cdot h''$ according to the decomposition of $\G^\der$ as an almost direct product $\G_I \cdot (\G_{I^c} \cdot \LL)$, and acts by $h\cdot (D_I \times \{t\}) = D_I \times \{h''t\}$ (see Remark \ref{splitaction}) which has to be equal to $D_I \times \{t'\}$ as both contain a point $x'$. This proves the above claimed equality of dimensions.

We are left to prove that up to replacing $S$ by some non-empty Zariski open subset, we have
\[
\dim\big(\I \cap (D_{I} \times \{t\})\big) = \dim\big(\I \cap (D_{I} \times \{t'\})\big)
\]
for any $t$ and $t'$ in the set $T \subset D_{I^c} \times D_L$ of second projections of $(\I,D_M,D_I)$-intersecting pairs. 

Let $\psi$ denote the composition of the lifted period map $\tilde{\Phi}_\F : \F \rightarrow D$ with the projection $p : D \rightarrow D_{I^c} \times D_L$, so that $T = \psi(\F)$. For $t \in T$, we denote the corresponding fibers by $\F_t := \psi^{-1}(t)$ and $\I_t := \I \cap p^{-1}(t) = \tilde{\Phi}_\F(\F_t)$. The map $\psi$ is complex analytic and definable. Therefore, if $C$ denotes the minimal dimension of a fiber $\F_{\psi(z)}$ for $z$ varying in $\F$, the set
\[
\F' := \{ z \in \F \hspace{0.1cm} | \hspace{0.1cm} \dim_z \F_{\psi(z)} > C \}
\]
is a strict definable closed analytic subset of $\F$.

By definability of $\exp$, it follows that the set \[S' := \exp(\F')\] is a strict definable subset of $S$. 

\begin{lemma}
The subset $S'$ is a strict Zariski-closed subset of $S$.
\end{lemma}
\begin{proof}
As $S$ is a quasi-projective complex variety and $S'$ is a strict definable subset of $S$, the o-minimal Chow theorem of \cite{PS} ensures that it suffices to show that $S'$ is actually a closed-analytic subset of $S$.

To see this, first note that $\exp \vert_\F$ is a local analytic isomorphism as it is a local property which is satisfied by definition on the $\F_i$'s (the maps are obtained as restrictions of a universal covering map to an open set). Futhermore, by definition the diagram
\[
\xymatrix{
\F \ar[r] \ar@/^2pc/[rrr]^\psi \ar[rd] & \tilde{S} \ar[r]^{\tilde{\Phi}} \ar[d] & D \ar[r]^{p_{I^c}} \ar[d] & D_{I^c} \ar[d]\\
& S \ar[r]_\Phi \ar@/_1.5pc/[rr]_{\Phi_{I^c}} & \Gamma \quo D \ar[r]_{q_{I^c}} & \Gamma_{I^c} \quo D_{I^c}
}
\]
commutes. Combining these two facts one sees that for any $s \in S$ and any lift $z \in \F$, one has $\dim_s \Phi_{I^c}^{-1}(\Phi_{I^c}(s)) = \dim_z \F_{\psi(z)}$ so that set theoretically $S'$ has the following description:
\[
S' = \{s \in S : \dim_s \Phi_{I^c}^{-1}(\Phi_{I^c}(s)) > C \}.
\] 
The map $\Phi_{I^c}$ being complex analytic, this shows that $S'$ is a closed complex analytic subset of $S$, and the argument given at the beginning of the proof shows that it is in fact Zariski closed as required.
\end{proof}

Now note that if $z \in \F-\F'$, the intersection $\I \cap (D_I \times \{\psi(z)\}) = \I_{\psi(z)}$ must have minimal dimension at $\tilde{\Phi}_\F(z)$ among the family $(\I_t)_{t \in T}$: otherwise, $\F_{\psi(z)}$ wouldn't have dimension $C$ at $z$ (i.e. minimal dimension among the family $(\F_t)_{t\in T}$), contradicting $z \notin \F'$. So, the quantity $\dim \I_{\psi(z)}$ doesn't depend on the choice of $z \in \F-\F'$ and up to replacing $S$ by its non-empty Zariski open subset $S-S'$, the initially wanted equality of dimensions holds for any couple $(t,t') \in T^2$. This finishes the proof.

\end{proof}
\begin{remark}\label{equid}
It follows from the proof of Lemma \ref{indepcap} that after replacing $S$ by its non-empty Zariski open subset $S - S'$ (and subsequent modifications, see Remark \ref{modif}), all the components of $(g \cdot D_M) \cap (D_I \times \{t\})$ (resp. $\I \cap \{D_I \times \{t\})$) have the same dimension, where $(g,t)$ is any $(\I, D_M, D_I)$-intersecting pair. Indeed, the first part of the proof in fact shows that $\dim_{x} (g \cdot D_M) \cap (D_I \times \{t\}) = \dim_{x'} (g' \cdot D_M) \cap (D_I \times \{t'\})$, where $x$ and $x'$ were arbitrary. For the second claim, this follows from the definition of $S'$ and the discussion at the end of the proof.
\end{remark}
\begin{proof}[Proof of Lemma \ref{reform}]
Let $I \subsetneq \{1, \cdots n\}$ be a strict set of indexes and $(g,t)$ be a $(\I,D_M,D_I)$-intersecting pair. Then both (i) and (ii) follow from the equality
\begin{eqnarray*}
d_\mon(\M,D_M) - d_{I}(\M,D_M) & = & \dim (g \cdot D_M)\cap D_H - \dim(g \cdot D_M) \cap (D_{I} \times \{t\})\\ 
&& + \dim \I - \dim \I \cap (D_{I} \times \{t\}) - \dim D_H + \dim D_{I} \\
& = & \dim p_{I^c}(g \cdot D_M) + \dim p_{I^c}(\I) - \dim D_{I^c} \\
& = & \dim p_{I^c}(g \cdot D_M) + \dim \Phi_{I^c}(S^\an) - \dim D_{I^c}.
\end{eqnarray*}
Here we are implicitly using that $\dim \I - \dim \I \cap (D_I \times \{t\}) = \dim p_{I^c}(\I)$ (resp. $\dim (g \cdot D_M)\cap D_H - \dim(g \cdot D_M) \cap (D_{I} \times \{t\}) = \dim p_{I^c}(g\cdot D_M)$). These equalities use the equidimensionality remarked above.
\end{proof}





We can now turn to the proof of our main results. First note that our general criterion implies the one announced in the introduction for the $\Q$-simple case as explained in the following subsection.

\subsection{Proof of Theorem \ref{mainqsimple} assuming Theorem \ref{main}}
Let $(\M,D_M)$ be a strict Hodge sub-datum of $(\G,D)$. It is immediate from the definitions and Convention \ref{conv} that under the $\Q$-simplicity assumption on $\mon$ along with the assumption that $\mon = \G^\der$, factorwise (resp. strong) $\V$-admissibility is equivalent to (resp. strong) $\V$-admissibility. Moreover, under the assumption $\mon = \G^\der$, the datum $(\M,D_M)$ is automatically full on the $\LL$-factor. Hence, Theorem \ref{mainqsimple} is simply a reformulation of Theorem \ref{main} in the $\Q$-simple monodromy case.

\subsection{Proof of Theorem \ref{main} (i)}\label{pfmain}
\textit{Step 1 (Reduction to $\Phi(S^\an)$ smooth and complex analytic)}. 
First note that by Theorem \ref{algpermap}, the period map $\Phi$ factors through the analytification of some regular dominant map $f : S \rightarrow T$ between quasi-projective complex algebraic varieties. As a consequence, $f(S)$ contains a Zariski-dense Zariski-open subset $U$ of $T$. Denote by $U^\sm$ its smooth locus which is a non-empty Zariski open subset of $U$.



Then, up to replacing $S$ by the non-empty Zariski open subset $f^{-1}(U^\sm)$ of $S$ (which doesn't change the result, see Remark \ref{zariskiopen}), we can assume that $\Phi(S^\an)$ is smooth and locally closed complex analytic, which we do in the sequel. Let us at this point emphasize two consequences of these reductions that will be useful to us. First, with these changes, the image $\I$ of the fundamental set $\F$ by $\tilde{\Phi}_\F$ is a locally closed smooth complex analytic subset of $D$, and we can freely talk about complex analytic irreducible components of intersections with other complex analytic subvarieties of $D$. Furthermore, if $(\M,D_M)$ is a strict Hodge sub-datum of $(\G,D)$, then $D_M$ and $\I$ intersect with expected dimension along some complex analytic irreducible component $U$ inside of $D$ if and only if the equality of dimension
\[
\dim U = \dim \I + \dim D_M - \dim D
\]
holds (without the above argued smoothness, one might in principle only get an inequality in view of Definition \ref{expected}).
We will insist on this in the sequel by saying that $D_M$ and $\I$ intersect inside $D$ with \textit{the} expected dimension along $U$.



\textit{Step 2 (Localisation of the problem)}. Let $(\M, D_M)$ be a factorwise $\V$-admissible strict Hodge sub-datum of $(\G,D)$. Let $s \in S^\an - \HL$ be a Hodge generic point. Fix some $z \in \I$ such that $\pi(z) = \Phi(s) =: x$. As $G := \G(\R)^+$ acts transitively on $D$, there exists an element $g \in G$ such that $z \in g \cdot D_M$. Let $U$ be an irreducible complex analytic component of the intersection of $g \cdot D_M$ and $\I$ 
containing $z$. We want to recall here how, following an idea of Chai \cite[Prop. 1]{chai} (see also \cite[Sect. 10.1]{bku}), one can reduce the proof of the density of the Hodge locus to producing a component $U \subset (g \cdot D_{M}) \cap \mathcal{I}$ which has the expected dimension $d = \dim D_M + \dim \I - \dim D$. 

Suppose we have done this. Then in some neighbourhood $\mathcal{N}$ of a smooth point $u \in U$ the intersection between $g \cdot D_{M}$ and $\mathcal{I}$ is a transverse intersection of smooth manifolds inside $D$. Such intersections are stable under small smooth perturbation (see for instance \cite[Ch. 1, \S6, Ex. 11]{zbMATH03562121}), so one obtains an open neighbourhood $\mathcal{V} \subset G := \G(\R)^+$ such that, for each $g \in \mathcal{V}$, the intersection $(g \cdot D_{M}) \cap \mathcal{I} \cap \mathcal{N}$ is non-empty. By \cite[Theorem 18.2]{lagbor}, $\G$ is unirational as an algebraic variety over $\Q$ so $\G(\Q)^+$ is dense in $G$. We may therefore choose $g_{0} \in \G(\Q)^+ \cap \mathcal{V}$ such that the translate $g_{0} \cdot D_M$ and $\I$ intersect in $D$ with the expected dimension $d$ along some complex analytic irreducible component intersecting $\mathcal{N}$. 

Now, pulling-back these components by the diagram (\ref{Idefdiag}) gives analytic germs of the Hodge locus of $\HLM$ meeting any sufficiently small analytic neighbourhood of the Hodge generic point $s$ in $S^\an$. By density of Hodge generic points for $\V$ in $S^\an$, this gives the desired property of density of the Hodge locus $\HLM$ of type $\M$. Let us emphasize at this point that concluding the density of the typical Hodge locus of type $\M$ as in (ii) requires more work, as the components we constructed above could in principle all be atypical with respect to their generic Mumford-Tate group.

\vspace{0.5em}

\textit{Step 3 (Using Ax-Schanuel to constrain the monodromy of atypical intersections)}.  Let us now prove that under our assumptions, $g\cdot D_M$ and $\I$ intersect inside $D$ with the expected dimension along $U$, that is
\begin{equation}
\label{typeq}
\codim_{D} U = \codim_{D} \I + \codim_{D} g \cdot D_M.
\end{equation}
Assume the opposite for the sake of contradiction. We claim that $\dim D - \dim g \cdot D_M = \dim D_H - \dim (g \cdot D_M) \cap D_H$ as a consequence of the fact that $(\M,D_M)$ is full on the $\LL$-factor.\footnote{Note that in the above equality $D_H$ is both the monodromy factor, and an orbit in $D$ under $H$. We are implicitly using Remark \ref{constfactorrem} and denoting some monodromy orbit $D_H \times \{z_L\}$ by $D_H$ where $z_L \in D_L$ is a lift of the $t_L \in \Gamma_\LL \quo D_L$ in \textit{loc. cit.} such that $z \in D_H \times \{z_L\}$.} Indeed, because $\LL$ is a normal $\Q$-subgroup of $\G^\der$, and $\LL \subset \M$, one finds that $L := \LL(\R)^+ \subset g M g^{-1}$ where $M = \M(\R)^+$. As $L$ acts transitively on $D_L$, it follows from the above containment that $\pr_{D_L}(g \cdot D_M) = D_L$ and the claimed equality of dimension follows from
\[
\dim g\cdot D_M = \dim \pr_{D_L}(g \cdot D_M) + \dim (g \cdot D_M) \cap D_H,
\]
which follows from the equidimensionality argued in Remark \ref{equid}. Therefore, as we have assumed (\ref{typeq}) does not hold, the intersection of $(g \cdot D_M) \cap D_H$ and $\I$ inside the monodromy orbit $D_H$ has bigger dimension than expected along $U$, i.e.
\begin{equation}\label{atyp}
\codim_{D_H} U < \codim_{D_H} \I + \codim_{D_H} (g \cdot D_M) \cap D_H.
\end{equation}
Consider the algebraic subvariety\footnote{Here we recall that an algebraic subvariety of $S \times D_H$ is understood to be a component of an intersection $(S \times D_H) \cap V$, where $V$ is an algebraic subvariety of $S \times \check{D}_H$, where $\check{D}_H$ is the compact dual.} $W = S \times ((g \cdot D_M) \cap D_H)$ of $S \times D_H$, let $\tilde{U}$ be a complex analytic component of $W \cap (S \times_{\Gamma_\mon \quo D_H} D_H)$ containing $(s,z)$ such that the germs of $\pr_{D_H}(\tilde{U})$ and $U$ at $z$ coincide. 
Then, the condition (\ref{atyp}) gives the inequality required in Theorem \ref{axschan}, which then ensures that the projection $\pr_S(\tilde{U})$ of $\tilde{U}$ to $S$ is contained in a strict weakly special subvariety $Y$ of $S$ for $\V$. We take $Y$ to be minimal for the inclusion among those weakly special subvarieties of $S$ containing $\pr_S(\tilde{U})$. Since $Y$ contains $s$ which is Hodge generic, it has generic Mumford-Tate group $\G$, so that, by André's Theorem \ref{andre}, the algebraic monodromy group $\mon_Y$ of $\V \vert_{Y}$ is a strict normal subgroup of $\mon$. Let us write $\mon_Y = \prod_{i \in I} \mon_i$. Applying Remark \ref{monperiodmap} to the restriction of $\Phi$ to the smooth locus of $Y$, one has that $Y^\an$ is contained in an analytic irreducible component of the preimage by $\Phi$ of some analytic subvariety of $\Gamma \quo D$ of the form $\pi(D_I \times \{t\})$. Furthermore, any irreducible subvariety of $S$ contained in the preimage by $\Phi$ of $\pi(D_I \times \{t\})$ must have algebraic monodromy group contained in $\mon_Y$. As a weakly special subvariety is by definition maximal for the inclusion among irreducible subvarieties with the same algebraic monodromy group, this shows that $Y^\an$ is an analytic irreducible component of the preimage by $\Phi$ of $\pi(D_I \times \{t\})$.

We may regard $U$ as a component of the intersection of $(g\cdot D_M) \cap (D_I \times \{t\})$ and $\I \cap (D_I \times \{t\})$ for some choice of $t$. We claim that these two intersect in $D_I \times \{t\}$ with the expected dimension along $U$. Indeed, assuming that they don't, we can apply as above the Ax-Schanuel theorem for $\V \vert_Y$ to the intersection $\tilde{U}$ to conclude that $\pr_S(\tilde{U})$ must be contained in some strict weakly special subvariety $Y'$ of $Y$ for $\V\vert_Y$. Note that $Y'$ is also a weakly special subvariety of $S$ for $\V$ as, via the same reasoning just exhibited for $Y \subset S$, it is the pullback via $\Phi$ of some translate of $\Gamma_{\mon_{Y'}} \quo D_{H_{Y'}} \times \{ t' \}$, and all such varieties are weakly special by the characterization given in Proposition \ref{caracweakly}. 
To sum up, $Y'$ is a weakly special subvariety of $S$ for $\V$, that contains $\pr_S(\tilde{U})$, and that is strictly contained in $Y$. This contradicts the minimality of $Y$, proving the intersection of $(g\cdot D_M) \cap (D_I \times \{t\})$ and $\I \cap (D_I \times \{t\})$ in $D_{I} \times \{ t \}$ has the expected dimension.

We then get the following chain of (in)equalities:
\begin{eqnarray*}
0 & = & \dim U - \dim\big(U \cap (D_{I} \times \{t\})\big) \\ & > & d_H(\M,D_M) - d_I(\M,D_M) \\
& \geqslant & 0,
\end{eqnarray*}
where the first line follows from the fact that $U$ lies in $D_I \times \{t\}$, the second follows from the assumption that $U$ has bigger dimension than expected and the fact that $U\cap (D_{I} \times \{t\})$ has the expected dimension, and the third one is factorwise $\V$-admissibility. This is a contradiction, so the inequality (\ref{atyp}) must fail and the reasoning presented at the beginning of the proof gives the desired density of $\HLM$.

\subsection{Proof of Theorem \ref{main} (ii)}
We keep the notations of the proof of Theorem \ref{main}(i) and assume now that $(\M,D_M)$ is factorwise strongly $\V$-admissible. The proof of (i) gives an open neighbourhood $\mathcal{V}$ of $g$ in $G$ such that for each $g_{0} \in \G(\Q)^+ \cap \mathcal{V}$, we have that $g_{0} \cdot D_M$ and $\I$ intersect along some complex analytic irreducible component with the expected dimension $d$. As explained above, this gives a set of components of $\HLM$, intersecting some neighbourhood of the fixed point $s \in S$ in some dense set, all with the expected dimension. However these components might, in principle, not belong to $\HLM_\typ$ on account of their Mumford-Tate groups being properly contained in some conjugate of $\M$. 

To rule this out we show that after removing a proper closed subset of $\mathcal{V}$ of smaller definable dimension the above constructed components of the Hodge locus of type $\M$ will have the desired property. To construct this closed subset we recall that Mumford-Tate domains in $D$ lie among the fibres of finitely many real-algebraic definable families $f_{k} : \mathcal{D}_{k} \to G$ for $1 \leq k \leq \ell$. Indeed, by \cite[VI.B.12.(i)]{ggk}, for any Mumford-Tate subgroup $\NN$ of $\G$, the domain $D$ contains finitely many $\NN(\R)^+$-orbits of points with Mumford-Tate group equal to $\NN$ and by \cite[Lemma 7.3]{gfunctions}, the group $\G$ contains finitely many $\G(\R)$-conjugacy classes of Mumford-Tate subgroups. More precisely, for each such $k$, there is a Mumford-Tate domain $D^{(k)}$ such that $\mathcal{D}_{k,g} := f_k^{-1}(g) = g \cdot D^{(k)}$, and all Mumford-Tate subdomains of $D$ arise as $\mathcal{D}_{k,g}$ for some choice of $k$ and $g$. We consider only those families for which some fibre of $f_{k}$ is a Mumford-Tate domain properly contained in $D_M$. After redefining the families, we can assume that $D^{(k)}$ is this fiber. In particular, the fibres of $f_{k}$ have strictly smaller dimension than $D_{M}$. For each such $k$, denote by $(\M^{(k)}, D^{(k)})$ the associated Hodge datum, and construct the locus
\[ \mathcal{P}_{k} = \{ (g', x) : x \in g' \cdot D^{(k)} \cap \mathcal{I}, \hspace{0.5em} \dim_{x} (g' \cdot D^{(k)} \cap \mathcal{I}) = d \} \subset G \times D. \]
We claim that the projection $\mathcal{G}_{k}$ of $\mathcal{P}_{k}$ to $G$ intersects $\mathcal{V}$ in a set whose closure has smaller definable dimension. Using definable cell decomposition \cite[Ch. 3, \S2.11]{zbMATH01160037}, it suffices to show that $\mathcal{G}_{k} \cap \mathcal{V}$ does not contain any non-empty open subset $\mathcal{B}$ of $G$. We suppose that it does, and we use this to pick $g' \in \mathcal{B}$ such that some irreducible component $U' \subseteq (g' \cdot D^{(k)}) \cap \mathcal{I}$ of the intersection contains the image $z'$ by $\tilde{\Phi}_\F$ of some lift to $\F$ of a Hodge generic point in $S$. 


Because $g' \in \mathcal{G}_k$, we have by definition that $U'$ has dimension $d$, and using the assumption that $D^{(k)}$ has strictly smaller dimension than $D_{M}$ it follows that $g' \cdot D^{(k)}$ and $\mathcal{I}$ intersect in $D$ with bigger dimension than expected along $U'$. As we assumed $\mon = \G^\der$, this implies that they intersect in $D_H$ with bigger dimension than expected along $U'$. Choose a component $\tilde{U}'$ of the intersection of $S \times g' \cdot D^{(k)}$ with $S \times _{\Gamma_\mon \quo D_H} D_H$ in $S \times D_H$ such that the germs of $\pr_{D_H}(\tilde{U}')$ and $U'$ at $z'$ coincide. Then the Ax-Schanuel Theorem \ref{axschan} gives a strict weakly special subvariety $Y' \subset S$ which contains the projection to $S$ of $\tilde{U}'$. 
Choose $Y'$ to be minimal for the inclusion among the weakly special subvarieties of $S$ for $\V$ containing $\pr_S(\tilde{U}')$. By André's theorem \ref{andre}, since $Y'$ contains a Hodge generic point, $\mon_{Y'}$ is a normal subgroup of $\G^\der$, hence of $\mon$. Write $\mon_{Y'} = \prod_{i \in I} \mon_i$. As in Step $3$ of the proof of (i), the variety $Y'$ is a component of the inverse image under $\Phi$ of $\pi(D_{I} \times \{ t \})$ for some choice of $t$. Applying the Ax-Schanuel theorem for $\V\vert_{Y'}$ to the intersection of $(g' \cdot D^{(k)}) \cap (D_I \times \{t\})$ and $\mathcal{I}$ in $D_I \times \{t\}$ along $U' = U' \cap (D_I \times \{t\})$, one proves as in Step $3$ of the proof of (i) that this intersection has expected dimension.

Now, because $g' \in \mathcal{V}$, $U'$ is also a component along which $(g' \cdot D_M) \cap D_H$ and  $\mathcal{I} \cap D_H$ intersect with expected dimension $d$ inside of $D_H$. Wrapping things up, we get the following chain of (in)equalities:
\begin{eqnarray*}
0 & = & \dim U' - \dim\big(U' \cap (D_{I} \times \{t\})\big) \\ & = & d_\mon(\M,D_M) - d_I(\M^{(k)},D^{(k)}) \\ & \geqslant & d_\mon(\M,D_M) - d_I(\M,D_M)  \\
& > & 0.
\end{eqnarray*}
Here, the first line follows from the fact that $U'$ lies in $D_I \times \{t\}$, the second from the fact that $g' \cdot D_M$ and $\I$ (resp. $g' \cdot D^{(k)} \cap (D_I \times \{t\})$ and $\I \cap (D_I \times \{t\})$) intersect with expected dimension inside of $D_H$ (resp. $D_I \times \{t\}$) along $U'$ (resp. $U' \cap (D_I \times \{t\})$, the third is immediate from definitions and the fourth is factorwise strong $\V$-admissibility. This is a contradiction, and we have proven that $\mathcal{G}_k$ intersects $\mathcal{V}$ in a set whose closure has strictly smaller definable dimension.

Now, by density of $\G(\Q)^+$ in $G$, we can pick some $g_1 \in \G(\Q)^+ \cap (\mathcal{V}-\mathcal{G})$ where $\mathcal{G} = \bigcup_k \mathcal{G}_k$. By construction, pulling back the intersection between $g_1 \cdot D_M$ and $\I$ to $S$ by the diagram (\ref{Idefdiag}) gives rise to components of the typical Hodge locus of type $\M$. Since we can take $g_1$ arbitrarily close to the $g$ we fixed at the beginning of the proof, we therefore can construct components of $\HLM_\typ$ intersecting any neighbourhood of the fixed Hodge generic point $s \in S$. Using the density of the Hodge generic points in $S$ finishes the proof.


\newpage
\printbibliography
\end{document}